\documentclass[12pt,letterpaper]{amsart}

\usepackage[utf8]{inputenc}

\usepackage{amssymb,latexsym,amsmath,amsthm,graphicx}
\numberwithin{equation}{section}
\newtheorem{thm}{Theorem}[section]
\newtheorem{lemma}[thm]{Lemma}

\newtheorem*{remark}{Remark}

\usepackage{xcolor}

\newcommand{\ol}{\overline}

\def \f{\tilde{f}}
\def \g{\tilde{g}}
\def \h{\tilde{h}}
\def \z{\tilde{\zeta}}
\def \F{\tilde{F}}
\def \p{\partial}

\begin{document}

\title[minimal graphs]{A sharp bound for the growth of minimal graphs}

\author[Allen Weitsman]{Allen Weitsman\\ \\
D\MakeLowercase{edicated with gratitude to the memory of}
W\MakeLowercase {alter} H\MakeLowercase{ayman}.
}

\address{Email: weitsman@purdue.edu}


\begin{abstract}
We consider minimal graphs $u = u(x,y) > 0$ over unbounded domains $D \subset R^2$ bounded by a
Jordan arc $\gamma$ on which
$u = 0$. We prove a sort of reverse Phragmén-Lindelöf theorem by showing that
if $D$ contains a 
sector  $$S_{\lambda}=\{(r,\theta)=\{-\lambda /2<\theta<\lambda /2\}\quad(\pi <\lambda \leq 2\pi),$$ then the
rate of growth
is at most   $r^{\pi/\lambda}$.
\vskip .1 truein
\noindent {\bf keywords.}  minimal surface, harmonic mapping, asymptotics

\end{abstract}
\maketitle
\section{Introduction} Let $D$ be an unbounded plane domain. 
  In this paper we consider the boundary value problem
for the minimal surface equation
\begin{equation}
\label{eq:bdryvalueprob}
\text{div} \frac{\nabla u}{\sqrt{1+|\nabla u|^2}}=0 
 \end{equation}
 with
 \begin{equation}
 \label{eq:bdryvalueprob1}
   u=0\quad \text{on}\ \partial D \quad \text{and } u>0 \quad \text{in}\ D.
 \end{equation}

 We shall study the constraints on growth of nontrivial solutions to (\ref{eq:bdryvalueprob})
 and (\ref{eq:bdryvalueprob1}) as determined
 by the maximum 
 $$
 M(r)={\sup}\ u(x,y),
 $$ 
where the sup is taken over the values $r=\sqrt{x^2+y^2}$ and $(x,y)\in D$.

The methods of this paper extend the results of \cite{LW}, where the following is proved.
\vskip .2 truein
\noindent
{\bf Theorem A.} 
 \emph{Suppose $D$ is a simply connected domain  whose boundary is a Jordan arc, and $D$ contains a sector $S_\lambda = \{ z: |\arg z| <\lambda /2\}$, with $\pi <\lambda \leq 2\pi $. 
 With $M(r)$ defined as above, if $u$ satisfies (\ref{eq:bdryvalueprob})
 and (\ref{eq:bdryvalueprob1}) in $D$, then there exist positive constants $K$ and $R$ such that}
\begin{equation}
\label{eq:bounds}
	 M(r) \leq K r, \quad r > R.
\end{equation}
\vskip .2truein
As in Theorem A above,  throughout this paper we shall use complex notation for convenience.

Results regarding upper and lower bounds for the growth of solutions to  (\ref{eq:bdryvalueprob})
and (\ref{eq:bdryvalueprob1}) are rather 
scarce and fragmented.  To begin with, the first relevant theorem in this direction was proved by Nitsche  \cite[p. 256]{Nitsche} who 
observed that there are no solutions to (\ref{eq:bdryvalueprob}) and (\ref{eq:bdryvalueprob1}) with
 $D$ being contained in a sector of opening strictly less than $\pi$.
For domains contained in a half plane, but not contained in any such sector, there
 are solutions to (\ref{eq:bdryvalueprob}) and (\ref{eq:bdryvalueprob1})  with differing growth rates given in \cite{LW}. 
 
   For angles $\lambda \geq
\pi$, in terms of the order
 $\rho$ of $u$ defined
 by
 $$
 \rho=\underset {r\to\infty}\lim\sup\frac{\log M(r)}{\log r},
 $$
 it follows 
 by using the module estimates of Miklyukov \cite{Miklyukov} 
 as in \cite{Weitsman2005a} that if $D$ omits a sector of opening $2\pi -\alpha $, ($\pi\leq\alpha\leq 2\pi)$,
the omitted set in the case $\alpha=2\pi$ being a line, then
the order $\rho$ of any nontrivial solution to (\ref{eq:bdryvalueprob}) and (\ref{eq:bdryvalueprob1})  is at least
$\pi/\alpha$,
More precisely, the results in \cite{Weitsman2005a} are phrased in terms of the asymptotic angle $\beta$ defined as follows.  

 Let $\Theta (r)$ be the angular measure of the set $D\cap\{|z|=r\}$ and
$$
\beta =\underset{r\to\infty}{\lim\ \sup} \ \Theta(r).  
$$
With this definition, the lower bound is given by

\vskip .2truein
\noindent
{\bf Theorem B. }  \emph{ Let $D$ be an unbounded domain whose boundary $\partial D$ is a piecewise differentiable arc,
and suppose that $u$ satisfies  (\ref{eq:bdryvalueprob}) and (\ref{eq:bdryvalueprob1}).  If $\beta \geq\pi$, then $ \rho \geq \pi /\beta$.}
\vskip .2truein
Regarding upper bounds,  it was conjectured \cite{Weitsman2005} that solutions to 
(\ref{eq:bdryvalueprob}) and (\ref{eq:bdryvalueprob1}) in general have at most exponential growth, and this is achieved by the 
horizontal catenoid.  In \cite{Weitsman2005} the following is proved.
\vskip .2truein\noindent

{\bf Theorem C. } \emph{If $u $ satisfies  (\ref{eq:bdryvalueprob}) and (\ref{eq:bdryvalueprob1}) in 
 a domain $D$ contained in a half plane and bounded by an unbounded Jordan arc,
then
$$
  Cr\leq M(r)\leq e^{Cr}\quad (r>r_0)
 $$
 for some positive constants $C$ and $r_0$. } 
 \vskip .2truein
 
The main result of this paper is the following bound for the order $\rho$ of solutions when $D$ contains
a large sector.
\begin{thm}\label{thm:bound}
 Let $D$ be a simply connected domain  whose boundary is a Jordan arc, and $D$ contains a sector $S_\lambda = \{ z: |\arg z| < \lambda /2\}$, with $\pi<\lambda\leq 2\pi $. 
If $u$ satisfies (\ref{eq:bdryvalueprob}) and (\ref{eq:bdryvalueprob1}) in $D$, then $\rho\leq \pi /\lambda$.
\end{thm}

The examples given in \cite{LW} show that the theorem is sharp.  Further details regarding those
prototypes can be found in \cite{Weitsman2020}.

Note that Theorem B and Theorem 1.1 taken together imply that if $ D$ is bounded by a 
piecewise differentiable arc and is asymptotic to  a sector $S_\lambda$ with
$\lambda>\pi$, then the order of $u$ will in fact be equal to $\pi/\lambda$. 

\section{PRELIMINARIES}

Let $u(z)$ be a solution to (\ref{eq:bdryvalueprob}) and (\ref{eq:bdryvalueprob1}) over a simply connected domain $D$.
We shall  make use of the parametrization of the surface given by $u$ in isothermal coordinates using 
Weierstrass functions $\left( x(\zeta), y (\zeta ), U (\zeta) \right)$
with $\zeta$ in the right half plane $H$ and $U(\zeta) = u(f(\zeta))$,
where
\begin{equation}
\label{downstairs}
z=f(\zeta) = x(\zeta) +iy(\zeta),\quad\zeta\in H.
\end{equation}
Then $f(\zeta )$ is univalent and harmonic, and since $D$ is simply connected it can be written in
the form 
\begin{equation}
\label{decomp1}
f(\zeta) = h(\zeta) + \ol{g(\zeta)}
\end{equation}
where $h(\zeta )$ and $g(\zeta)$ are analytic in $H$, 
\begin{equation*}
|h'(\zeta)|>|g'(\zeta)|,
\end{equation*}
and
\begin{equation}
\label{decomp2}
U(\zeta )= 2\Re e\, i\int \sqrt{h'(\zeta )g'(\zeta )}\,d\zeta  .
\end{equation}
(cf. \cite{Duren}).

Now, $U(\zeta )$ is harmonic and in (\ref{decomp2}) can be taken as
 positive in $H$ and vanishing on $\partial H$.  Thus 
(cf. \cite[p. 151]{Tsuji}), 
%
$$
U(\zeta )= C\,\Re e\, \zeta,
$$
where $C$ is a positive constant.  This with
(\ref{decomp2}) gives
$$
 h'(\zeta)g'(\zeta) = - C^2/4.
$$

We may reparametrize for convenience, replacing $f(\zeta)$ and $U(\zeta)$ by
 $f(2\zeta/C)$ and $U(2\zeta/C)$.  Continuing to use $\zeta$ as the preferred variable, this means
 we may assume that
\begin{equation}
\label{height}
h'(\zeta)g'(\zeta)=-1 \ \ \textrm{and}\ \ U(\zeta)=2\Re e\,\zeta .
\end{equation}

In order to estimate the function $f(\zeta)$ in (\ref{decomp1}),  we shall use the following
lemma on quasiconformal mappings from \cite{DW} (see \cite[Lemma 5.8]{Drasin}). 
\eject

{\bf Lemma A. }  \emph{Suppose $\varphi$ is quasiconformal in the plane such that $\varphi(\infty )=\infty$, 
 and the dilatation
 $$
 \mu(z)=\varphi_{\overline z}(z)/\varphi_z(z)
$$
satisfies 
\begin{equation}
\label{Dil}
\int_0^{2\pi} |\mu (re^{i\theta })|\, d\theta \to 0\ \qquad\ ( r\to\infty).
\end{equation}
Then, in any fixed annulus $A(R)= \{R^{-1}\leq |z|\leq R\}\ \ (R>1)$,
$$
\frac{\varphi (tz)}{\varphi (t)}\to z
$$
uniformly in $A(R)$ as $0<t\to \infty$.  In particular, }
$$
|\varphi (z)| = |z|^{(1+o(1))}\ \ \ (z\to \infty).
$$
\vskip .1truein
{\bf Remark.}  For our later applications of Lemma A, note that if $r>0$ and $a$  and $b$ are in $(0,2\pi)$,
then
$$
\frac{\varphi(tre^{ia})}{\varphi(tre^{ib})}=\frac{\varphi(tre^{ia})}{\varphi(t)}\frac{\varphi(t)}{\varphi(tre^{ib})}
\to\frac{re^{ia}}{re^{ib}}=e^{i(a-b)}
$$
as $t\to\infty$, uniformly in $A(R)$.
\vskip.1truein
At the last stage we shall need a barrier argument based on the following \cite[p.827]{Hwang88}.
\vskip .1truein
{\bf Lemma B.}  
Let $u(z)$ be a solution to the minimal surface equation over a domain $\Omega$ of the form
$S_\lambda
\backslash E\ \ (\lambda<\pi)$ with $u(z)=0$ on $\partial E$
and $u(z)\leq ax^m+b\ \ (0<m<1, a,\ b \geq 0)$ for $z=x+iy
\in\partial S_\lambda$.  Then 
$u(z)\leq ax^m+b$ in $\Omega$.
\begin{proof}  Let $T_1=S_\lambda \cap \{z:\Re e\, z<1\}$  and $A>0$.  Then, there exists 
\cite[p.322]{JS}
 a solution
$V_{1,A}(z)$ to the minimal surface equation over $T_1$ with values $Ax$ on\newline  
$\partial S_\lambda\cap\{z:\Re e\, z<1\}$,  and $V_{1,A}(z)\to+\infty$ uniformly if
$\Re e\, z \to 1$ and $|\arg z|\leq\lambda'/2 <\lambda/2$ in $T_1$.  The dilations $V_{R,A}(z)=RV_{1,A}(z/R)$ have 
corresponding properties for $T_R=S_\lambda\cap \{z:\Re e\, z <R\}$.  Now, by the max/min
principle \cite[p.94]{Osserman},  $V_{R,A}(z)>Ax$ for 
$z\in S_\lambda\cap\{z:\Re e\,z<R\}$ and
 $\{V_{R,A}\}$ decreases with
$R$  on compact subsets of $S_\lambda$.  Thus, by the  monotone convergence theorem \cite[p.329]{JS}, 
 $V_{R,A}\to V_A$ on $S_\lambda$, where $V_A$ is a 
solution to the minimal surface equation with boundary values $Ax$ in a sector of opening less than
$\pi$.  Therefore  \cite[p.256]{Nitsche}, $V_A(z)\equiv Ax$ for $z=x+iy\in S_\lambda$. 

Let $u(z)$ be as in the statement of the lemma and take a fixed  $x_0>0$.  Then, for $z\in \partial S_\lambda$,
\begin{equation}\label{linear}
u(z)\leq a\left (x_0^m+mx_0^{m-1}(x-x_0)\right ) +b.
\end{equation}
Since $u(z)=0$ on $\partial E$, it follows that
\begin{equation}\label{linear1}
u(z)\leq V_{R,A}(z)+B\quad z\in \Omega\cap\{\Re e\,z<R\},
\end{equation}
where $A=amx_0^{m-1}$, and $B=ax_0^m(1-m)+b$.   Letting $R\to\infty$ in (\ref{linear1}),  it then
follows that (\ref{linear}) holds in $\Omega$.  Thus for any $z=x_0+iy$ in $\Omega$ we have 
$$
u(x_0+iy)\leq ax_0^m+b,
$$
and since $x_0$ was arbitrary, the lemma is proved. 
\end{proof}
As a final preliminary lemma, we need the following qualitative growth estimate.
\begin{lemma}\label{preliminarygrowth}
Let $u(z)$ be a solution to (\ref{eq:bdryvalueprob}) 
and (\ref{eq:bdryvalueprob1}) over a domain $D$ containing a sector $S_\lambda$
with $\lambda >\pi$, and $f(\zeta),\ h(\zeta)$,  $g(\zeta )$, and $U(\zeta)$ 
be as in (\ref{decomp1}) and  (\ref{decomp2}) corresponding to $u(z)$. Then, for any proper subsector 
$S_{\lambda' }$ with $\pi<\lambda'<\lambda$ and $D_{\lambda'}=f^{-1}(S_{\lambda'})$,
$$
h'(\zeta)\to\infty\text{ as }\zeta\to\infty
$$
uniformly for $\zeta \in D_{\lambda'}$.
\end{lemma}
\begin{proof}

Let $f(\zeta)$, $U(\zeta)$ be as above.
So, $u(f(\zeta)) = U(\zeta) = 2\Re e\, \zeta $.

Let $P_\alpha= \{ \zeta: \Re e\  e^{i\alpha}f(\zeta) > 0 \}\ \  (|\alpha |<\lambda /2-\pi/2)$ 
and introduce a new variable $\z$, and let $\zeta =\psi_0(\z)$ be a conformal map 
from the right half plane $H = \{\z: \Re e \,\z>0 \}$ onto $P_0$ with $\psi_0(\infty)=\infty$.

Define
\begin{equation*}
\left\{
\begin{array}{l}
\f(\z) = f(\psi_0(\z)) \\
\g(\z) = g(\psi_0(\z)) \\
\h(\z) = h(\psi_0(\z))
\end{array}\right. 
\end{equation*}

Then $\f$ is a harmonic map 
$$ \f(\z) = \h(\z) + \ol{\g(\z)},\quad (\tilde\zeta\in H)$$
satisfying
\begin{equation}
\label{dil}
|\tilde h'(\tilde \zeta )|>|\tilde g'(\tilde\zeta)|, \quad (\tilde\zeta\in H).
\end{equation}
Note $\F(\z) = \h(\z) + \g(\z)$ is an analytic function with the same real part as $\f$. 
Then $\Re e \F$ is positive in $H$ and vanishes on $\p H$, and therefore
 (see \cite[p. 151]{Tsuji})
$$
 \F(\z) = k\z+ik_0 \implies \F'(\z) =k,\qquad k>0,
$$
that is,
\begin{equation}
\label{F1}
\tilde h'(\z)+\tilde g'(\z)=k>0.
\end{equation}
Now, 
\begin{equation}
\label{equationh'}
 \h'(\z) = h'(\psi_0(\z)) \cdot \psi_0'(\z), 
 \end{equation}
and by (\ref{height}),
\begin{equation}
\label{eq:g'}
 \g'(\z) = - \frac{\psi_0'(\z)}{h'(\psi_0(\z)) } = - \frac{\psi_0'(\z)^2}{\h'(\z)}. 
\end{equation}
Combining this with (\ref{F1})) we have
\begin{equation*}
k = \h'(\z) - \frac{\psi_0'(\z)^2}{\h'(\z)}  
\end{equation*}
which implies
\begin{equation*}
 \h'(\z)^2 - k\h'(\z) - \psi_0'(\z)^2 = 0.
\end{equation*}
Thus,
\begin{equation}  \label{eq:h'} 
\h'(\z) = \frac{k\pm  \sqrt{k^2+4\psi_0'(\z)^2}}{2},\ 
\  \g'(\z)=\frac{-2\psi_0'(\z)^2}{k\pm  \sqrt{k^2+4\psi_0'(\z)^2}}.
\end{equation}

Since $\psi_0(\z)$ is a conformal map with $\Re e\, \psi_0(\z) > 0$ in $H$, there exists a real constant
$0 \leq c < \infty$ such that in any sector $S_\beta = \{ \z: |\arg \z | \leq \beta < \pi/2 \}$,
  $\psi_0'(\z) \rightarrow c$ uniformly as $\z \rightarrow \infty$ in $S_\beta$ (see \cite[p. 152]{Tsuji}).
Thus,
\begin{equation}
\label{tildeforms}
 \h'(\z) \rightarrow \frac{k\pm\sqrt{k^2+4c^2}}{2},\quad
 \g'(\z) \rightarrow \frac{-2 c^2}{k\pm\sqrt{k^2+4c^2}}, \qquad
 \zeta /\tilde\zeta \to c.
\end{equation}

If $c>0$, a simple calculation
with (\ref{tildeforms}) shows that if the minus sign in (\ref{tildeforms}) were to hold,  this
 would contradict (\ref{dil}).  In case $c=0$, with a minus sign in  (\ref{tildeforms}),
 this would imply that $\tilde h'(\z)\to 0$.
 However, (\ref{dil}) and (\ref{F1}) show that this is not possible. 
 
 Thus,  (\ref{tildeforms}) becomes
 \begin{equation}
\label{tildeforms1}
 \h'(\z) \rightarrow \frac{k+\sqrt{k^2+4c^2}}{2},\quad
 \g'(\z) \rightarrow \frac{-2 c^2}{k+\sqrt{k^2+4c^2}}, \qquad
  \zeta /\tilde\zeta \to c.
\end{equation}
 
 {\bf Case 1:} $\psi_0'(\tilde \zeta)\to c>0$ as $\tilde\zeta\to\infty$ in $S_\beta$.

Using (\ref{tildeforms1}) we have
\begin{equation}
\label{aa}
 \h(\tilde\zeta) + \ol{\g(\tilde\zeta)} =\left[ k\Re e\, \z + i  \sqrt{k^2+4c^2}\,
 \Im m\, \z\right](1+o(1))
\end{equation}
as $\tilde\zeta\to \infty$ uniformly in $S_\beta$.  From this it follows that
\begin{equation}
\label{asymptotic}
f(\zeta ) =\left[k \Re e\, \zeta /c + i \sqrt{k^2+4c^2}\,
 \Im m\, \zeta /c\right](1+o(1)))
 \end{equation}
uniformly as $\zeta \to \infty$ in proper subsectors of $H$. 
Therefore,  $P_0$ is asymptotically the half plane $H$, in the sense that for any $0<\delta<\pi$,  we have
$P_0\supset S_{\delta}\cap\{|\zeta |>R\}$ for large $R$.

By (\ref{downstairs}) and (\ref{height}),  the graph of the minimal surface is given parametrically 
by $(\Re e\, f(\zeta),\Im m\, f(\zeta),
2\Re e\,\zeta )$.  Using (\ref{asymptotic}) we then have that the surface is asymptotic to a plane
as $\zeta\to\infty$ in proper subsectors of $H$.

Consider now $P_\alpha$ as above with $\alpha\ne 0$ and let  $\psi_\alpha(\z)$ be a conformal map 
from the right half plane $H = \{\z: \Re e \,\z>0 \}$ onto $P_\alpha$ with $\psi_\alpha(\infty)=\infty$.
In this case we define
\begin{equation}
\label{alpha}
\left\{
\begin{array}{l}
\f_\alpha(\z) = e^{i\alpha}f(\psi_\alpha(\z)) \\
\g_\alpha(\z) = e^{-i\alpha}g(\psi_\alpha(\z)) \\
\h_\alpha(\z) =  e^{i\alpha}h(\psi_\alpha(\z))
\end{array}\right. 
\end{equation}
Proceding in analogy with (\ref{F1})-(\ref{asymptotic}), we have

\begin{equation}
\label {roots}
 \h_\alpha'(\z)^2 - k_\alpha\h_\alpha'(\z) - \psi_\alpha'(\z)^2 = 0
\end{equation}
and with the principal branch of the square root,
\begin{equation}  \label{eq:h'1} 
\h_\alpha'(\z) = \frac{k_\alpha +  \sqrt{k_\alpha^2+4\psi_\alpha'(\z)^2}}{2},\ 
\  \g'_\alpha(\z)=\frac{-2\psi_\alpha'(\z)^2}{k_\alpha +  \sqrt{k_\alpha^2+4\psi_\alpha '(\z)^2}}
\end{equation}
in  $H$, where the minus sign in the roots of (\ref{roots}) is eliminated as before.
  
 With $S_\beta = \{ \z: |\arg \z | \leq \beta < \pi/2 \}$, again $\psi'_\alpha(\z)\to c_\alpha\geq 0$ as
$\z\to\infty \in S_\beta$.
We wish to show that when $c_0>0$ then $c_\alpha>0$.

Suppose that $c_\alpha=0$.  We reflect $\tilde f_\alpha$
to the left half $\tilde \zeta$ plane and note that the $\beta$ corresponding to the sectors $S_\beta$ can
approach $\pi/2$.   It then follows from (\ref{eq:h'1}) that 
 Lemma A applies to $\tilde f_\alpha(\z)$.  To apply Lemma A, we note that the dilatations in $S_\beta$ tend to 0, and since $\tilde f_\alpha$ is sense preserving, the dilatations are less than $1$ outside.  So (\ref{Dil}) applies.
 
  Then for $\epsilon>0$ (see Remark following Lemma A),  
 the image $\tilde f(S_\beta)$
 covers $S_{\pi -\varepsilon}\cap \{z:|z|>R\}$
  if $\beta$ is sufficiently close to $\pi/2$ and $R$ sufficiently large.  From this and (\ref{aa}) 
it follows that if $Q=Q_{\alpha,\beta}=\tilde f(S_\beta)
\cap e^{-i\alpha}\tilde f_\alpha (S_\beta )$, then for $\beta$ close to $\pi/2$ and all large $R$,  the intersection
$Q\cap\{z: |z|=R\}$ is nonempty.

From the original analysis of $P_0$, we find that for points $\zeta\in f^{-1}(Q)$, (\ref{equationh'}) 
and (\ref{tildeforms1}) imply that $h'(\zeta )$ remains bounded.

On the other hand, from the analysis of $P_\alpha$, it follows from (\ref{alpha})
and (\ref{eq:h'1})  that $g'(\zeta)/
h'(\zeta)\to 0$ as $\zeta\to \infty$ and $ \zeta\in f^{-1}(Q)$.  This with (\ref{height}) implies that
$h'(\zeta)$ is unbounded, a contradiction. 

  Therefore, it must be that
 $c_\alpha>0$ also, and as in the case of $P_0$ above,  the graph above $f(P_\alpha)$ must
be asymptotically a plane.  Since $f(P_0)$ and $f(P_\alpha)$ overlap, these graphs must be asymptotically
the same plane, and since $f(P_0)\cup f(P_\alpha)$ extends outside a half plane with $u(z)>0$, we 
obtain a contradiction.   We conclude that Case 1 cannot occur.
\eject
{\bf Case 2:} $\psi_0'(\tilde \zeta)\to 0$ as $\tilde\zeta\to\infty$ in $S_\beta$.

As in Case 1 above, Lemma A can be used to show that for $\epsilon>0$, 
 the image $\tilde f(S_\beta)$
 covers $S_{\pi -\varepsilon}\cap \{z:|z|>R\}$ for large $R$  if $\beta$ is sufficiently close to $\pi$.  Using  an argument
 similar to Case 1, we can then deduce that the $c_\alpha$ corresponding to each $P_\alpha\ \ (|\alpha|<
 \lambda/2-\pi/2)$ 
 must also be $0$. 
 Since  $S_{\lambda'}$, with $\pi<\lambda'<\lambda$, can be covered by the union of the $\tilde f(S_\beta )$ corresponding to $P_0,\ P_\alpha,\ P_{-\alpha}$ for some $0<\alpha<\lambda/2-\pi/2$ and large $R$,
  it follows
 from  (\ref{alpha}), and (\ref{eq:h'1}) that
 $g'(\zeta)/
h'(\zeta)\to 0$.  This with (\ref{height}) implies that
$
h'(\zeta)\to\infty$ uniformly as $\zeta\to\infty$ with $\zeta\in f^{-1} (S_{\lambda'}).
$
\end{proof}

\section{Proof of Theorem \ref{thm:bound}}
\begin{proof}
For fixed $\lambda$, let $f_1(\zeta)$ denote the function in (\ref{downstairs}) corresponding to a solution
to (\ref{eq:bdryvalueprob}) 
and (\ref{eq:bdryvalueprob1}) over a domain $D$ containing $S_\lambda$.   Then for $\lambda'$ such that 
$\pi<\lambda'<\lambda$
 we define $f_2(\zeta )=
 \zeta^{\lambda'/\pi} +1$. 
  Let $\tilde S_{\lambda'} =f_2(H)$ and
$\tilde H = f_1^{-1}(\tilde S_{\lambda'} )$.  Then if $\psi(\zeta)$ is a $1-1$ conformal mapping of $H$
onto $\tilde H$ with $\psi(\infty)=\infty$,  it follows that $f_1(\psi (H))=f_2(H)$ and there exists an orientation preserving
homeomorphism $\varphi :H\to H$ with $\varphi(\infty)=\infty$ such that
\begin{equation}
\label{phidefinition}
f_1(\psi(\zeta ))=f_2(\varphi (\zeta )),\quad \zeta\in H.
\end{equation}

Differentiating (\ref{phidefinition}) with respect to $\zeta$ and $\overline\zeta$, and using
the first equality in (\ref{height}) we obtain
\begin{equation}
\label{decomp4}
\psi'(\zeta)h_1'(\psi(\zeta))=\varphi_\zeta(\zeta)f_2'(\varphi(\zeta))
\end{equation}
and
\begin{equation}
\label{decomp5}
-\overline{\frac{\psi'(\zeta)}{h_1'(\psi(\zeta))}}=\varphi_{\overline \zeta}(\zeta)f_2'(\varphi(\zeta)).
\end{equation}
Dividing (\ref{decomp5}) by (\ref{decomp4}) we have
\begin{equation}
\label{decomp6}
\frac{1}{|h_1'(\psi(\zeta))|^2}=
\left|\frac{\varphi_{\overline\zeta}(\zeta)}{\varphi_{\zeta}(\zeta)}\right|.
\end{equation}
Now, $\psi(\zeta)\to\infty$ as $\zeta\to\infty$ in $H$, so by Lemma \ref{preliminarygrowth} it follows
that the left side of (\ref{decomp6}) tends to $0$.  
 
It therefore follows from (\ref{decomp6}) and the fact that $\varphi$ is a 
sense preserving differentiable homeomorphism,  that $\varphi$ is quasiconformal
in $H$ and  that the dilatation of $\varphi$ satisfies
\begin{equation}
\label{dilatation}
\left |\frac{\varphi_{\overline \zeta}(\zeta)}{\varphi_\zeta (\zeta)}\right |\to 0, \qquad (\zeta\to\infty,\ \ \zeta\in H).
\end{equation}
The mapping $\varphi$ can then be extended by reflection to a quasiconformal mapping of the 
complex plane onto the complex plane with (\ref{dilatation}) still in force.  
As in the proof of Lemma 2.1, Lemma A is applicable 
 to $\varphi$.  Further, by the symmetry of the reflection, the conclusion of   
Lemma A can be improved to
$$\varphi(re^{i\theta})=r^{(1+o(1))}e^{i(\theta +o(1))}$$
 so that
$$
f_1(\psi(re^{i\theta}))=f_2(\varphi(re^{i\theta}))=
r^{(\lambda'/\pi +o(1))} e^{i(\lambda'\theta/\\\pi+o(1))},\quad (\zeta=re^{i\theta}\to\infty,\quad\zeta\in H).
$$ 
\vskip .1truein
From this we see that, given any $\lambda''$ such that $\pi<\lambda''<\lambda'$, there is a proper
sector $\Sigma_{\lambda''}$ in $H$ such that $f_1(\psi(\Sigma_{\lambda''}))$ covers $S_{\lambda''}
\cap\{|z|>R\}$ for large $R$.
But $\psi(\zeta)$ is a conformal mapping of $H$ into $H$, so $\psi'(\zeta)\to k$ as $\zeta\to \infty$
in $\Sigma_{\lambda''}$  for some $k\geq 0$ (cf. \cite[p. 151]{Tsuji}).
  Combining this with (\ref{height}) we conclude that for sufficiently large $z$,
\begin{equation}
\label{Sigmabound}
u(z)<|z|^{(\pi/\lambda'+o(1))},\quad z\in S_{\lambda''}.
\end{equation}
The boundary of the sector $S_{\lambda''}$ on which (\ref{Sigmabound}) holds forms an
angle in the left half plane of opening less than $\pi$.  On the boundary
of $D$ in the left half plane $u(z)=0$.  Therefore, Lemma B applied to the sector of opening $2\pi-\lambda''$
centered on the negative real axis, with $E$ being the complement of $D$, tells us that (\ref{Sigmabound})
holds in $D\backslash S_{\lambda''}$.  Thus we see that (\ref{Sigmabound}) holds on all of $D$.
 Since $\lambda'$
can be taken arbitrarily close to $\lambda$ in (\ref{Sigmabound}),  the proof is
complete.
\end{proof}
\bibliographystyle{amsplain}

\end{document}